\documentclass[11pt,reqno]{amsart}
\usepackage{amsmath,amsxtra,latexsym,amsthm,amssymb,amscd,pb-diagram}
\usepackage{mathrsfs,mathabx,amsfonts}
\usepackage[colorlinks=true,linkcolor=red,citecolor=blue]{hyperref}
\usepackage[margin=1in]{geometry}

\usepackage{bm}
\usepackage{color}
\usepackage{makecell,multirow,diagbox}
\usepackage{mathtools}
\usepackage[displaymath, mathlines]{lineno}

\usepackage{graphicx}
\usepackage{epsfig}
\usepackage{array}
\usepackage{enumitem}

\newtheorem{definition}{Definition}[section]
\newtheorem{theorem}{Theorem}[section]
\newtheorem{proposition}{Proposition}[section]
\newtheorem{lemma}{Lemma}[section]

\newtheorem{remark}{Remark}[section]
\newtheorem{example}{Example}[section]

\begin{document}
\title[Semilinear damped wave equation with Riesz potential-type power nonlinearity]{The application of decay character on the global behavior of damped wave equation with Riesz potential-type power nonlinearity}

\thanks{$^* $\textit{Corresponding author:} Phan Duc An (anpd@hvnh.edu.vn)}

\maketitle
\centerline{\scshape Trung Loc Tang$^{1}$, Dinh Van Duong$^{2}$, Duc An Phan$^{3,*}$}
\medskip
{\footnotesize
\centerline{$^1$ Department of Mathematics, Thang Long University}
	\centerline{Nghiem Xuan Yem, Hoang Mai, Hanoi, Vietnam}

	\centerline{$^2$ Faculty of Mathematics and Informatics, Hanoi University of Science and Technology}
	\centerline{No.1 Dai Co Viet road, Hanoi, Vietnam}

    \centerline{$^3$ Department of Mathematics, Banking Academy of Vietnam}
	\centerline{No.12 Chua Boc, Kim Lien, Hanoi, Vietnam}
    }

\begin{abstract}
    In this paper, our first objective is to investigate the decay rates and the global-in-time existence of solutions to the semilinear damped wave equation with the Riesz potential-type power nonlinearity $\mathcal{I}_\gamma\left(|u|^p\right)$, where $\gamma\in[0,n)$, in terms of the decay character of the initial data. This approach enables us to establish global existence results for several classes of initial data. Our second objective is to show, via a blow-up argument, that the conditions imposed on the nonlinearity in the global existence theorem are sharp for initial data belonging to the pseudo-measure space $\mathcal{Y}^q$. As a consequence, we derive the new critical exponent
$$
p_{\mathrm{crit}}(n,q,\gamma):=1+\frac{2+\gamma}{n-q}
$$
for $1\leq n\leq 4$ and $0\leq \gamma<q<n/2$. Furthermore, we establish a sharp lifespan estimate for solutions that blow up in finite time.
\end{abstract}


\section{Introduction}
In this paper, we investigate the following Cauchy problem for semilinear damped wave equation with a Riesz potential term coupled with a power-law nonlinearity:
\begin{align}\label{Semilinear_Damped_Waves}
\begin{cases}
u_{tt}(t,x)-\Delta u(t,x)+u_t(t,x)=\mathcal{I}_\gamma(|u|^p)(t,x), &x\in \mathbb{R}^n, t>0,\\
u(0,x)=\varepsilon u_0(x),\ u_t(0,x)=\varepsilon u_1(x),&x\in \mathbb{R}^n,
\end{cases}
\end{align}
where $\gamma \in [0, n), \,p>1$, and $\varepsilon>0$ describes the size of initial data. The Riesz potential $\mathcal{I}_\gamma$ is defined by
$$
\mathcal{I}_\gamma (f)(t,x):= \frac{\Gamma\left(\frac{n-\gamma}{2}\right)}{2^\gamma \pi^\frac{n}{2} \Gamma\left(\frac{\gamma}{2}\right)}\left(|x|^{-(n-\gamma)} \ast_x f(t,x)\right)= \frac{\Gamma\left(\frac{n-\gamma}{2}\right)}{2^\gamma \pi^\frac{n}{2} \Gamma\left(\frac{\gamma}{2}\right)} \int_{\mathbb{R}^n} \displaystyle\frac{f(t,y)}{|x-y|^{n-\gamma}} d y,
$$
for any $f \in L_{\rm loc}^1\left(\mathbb{R}^n\right)$.
More generally, the Riesz potential $\mathcal{I}_\gamma$ can be interpreted as the inverse operator of the fractional Laplacian in the sense that $$\mathcal{I}_\gamma (f)(t,x) = (-\Delta)^{-\frac{\gamma}{2}} f(t,x)=\mathfrak{F}^{-1}(|\xi|^{-\gamma} \widehat{f}(t,\xi))(t,x).$$ Here, we denote $\widehat{f}(t,\xi):= \mathfrak{F}_{x\rightarrow \xi}(f(t,x))$ as the Fourier transform with respect to the spatial variable of a function $f(t,x)$ and $\mathfrak{F}^{-1}$ represents the inverse Fourier transform. For a more comprehensive account of these fundamental properties of the Riesz potential, we refer the reader to \cite{Landkof1972, Stein1970}.

Regarding the corresponding linear problem of \eqref{Semilinear_Damped_Waves} as follows:
\begin{equation}\label{Linear_Damped_Waves}
\begin{cases}
v_{tt}(t,x)-\Delta v(t,x)+v_t(t,x)=0, &x\in \mathbb{R}^n, \quad t>0,\\
v(0,x)=v_0(x),\quad v_t(0,x)=v_1(x),&x\in \mathbb{R}^n.
\end{cases}
\end{equation}
The pioneer result has been established by Matsumura in \cite{Matsumura1976}. Their main tools are Fourier splitting method to prove some basic decay estimates of solution to \eqref{Linear_Damped_Waves} and its derivatives. Moreover, the solution to \eqref{Linear_Damped_Waves} witnesses the diffusion phenomenon, that is, it shares the same asymptotic behavior of solution to the heat equation (see \cite{Radu2011,DabbiccoEbert2014}). Specifically, they obtain the approximation of the solution $u$ by the Gauss kernel when $t$ is large.

When the right hand side of \eqref{Semilinear_Damped_Waves} is $|u|^{p}$ (the special case $\gamma=0$ of Riesz potential nonlinearity), under the additional regularity $L^1$ for the initial data, the authors in \cite{Todorova2001, IkehataTanizawa2005,Ikehata2002} show that the critical exponent of this problem coincides with the Fujita exponent $p_{\mathrm{Fuj}}(n):=1+2/n$, by proving the global (in time) existence of small data solutions when $p > p_{F}(n)$ and the blow-up result for weak solutions in finite time even for small data if $1 < p < p_{F}(n)$. In addition, the paper \cite{Zhang-2001} showed that the value $p= p_F(n)$ belongs to the blow-up range.  Here, the critical exponent is understood as the threshold
between the global (in time) existence of small data solutions and the blow-up of solutions
even for small data. For the blow-up range $p \leq p_F(n)$, according to the works \cite{LiZhou1995, Ikeda2016, Lai2019}, the sharp lifespan estimates for blow-up solutions to (\ref{Semilinear_Damped_Waves}) for $\gamma = 0$ in all spatial dimensions have been investigated. Here, we denote by $T_{\varepsilon}$ the lifespan of solution in the following sense:
$$
\begin{aligned}
T_{\varepsilon}:=\sup\{& T > 0:\text{ there exists a unique local (in time) solution $u$ on $[0,T)$}\\
&\text{ with a fixed parameter $\varepsilon>0$}\}.
\end{aligned}
$$ 
Consequently, these papers provided the following sharp lifespan estimates:
$$
T_{\varepsilon} \sim \varepsilon^{-\frac{2(p-1)}{2-n(p-1)}} \text{ if } 1<p<p_F(n) \,\,\text{ and }\,\, \log(T_{\varepsilon}) \sim  \varepsilon^{-\frac{2}{n}} \text{ if } p = p_F(n).
$$

On the other hand, to study the uniform decay rates of solutions to dissipative evolution equations with initial data in $L^2(\mathbb{R}^n)$ spaces, Bjorland and Schonbek in \cite{Bjorland2009} introduce a new notion $r^*(u)$ associated with $u\in L^2(\mathbb{R}^n)$, called the decay character of $u$. This notion is to measure the singularity of $\widehat{u}$ near $0$. This idea has been applied to various dissipative equations, since the solutions of these equations exhibit the same asymptotic behavior as the solutions to the corresponding heat equations. See, for instance, \cite{Niche2015} for results on the compressible approximation of the Navier–Stokes equations, \cite{Armando2022} for the damped wave equation, \cite{AnhTrang2018} for the Camassa–Holm equations, and \cite{AnhDuongToi2026} for the generalized magnetohydrodynamic (MHD) equations.  

This idea has also been applied to structurally damped $\sigma$-evolution equations (see \cite{ADL2024,ADL2025}). Moreover, in these papers, the authors showed that the conditions imposed on the nonlinearity exponent in the global existence theorem are sharp in the $L^m$ framework, under suitable assumptions on the spatial dimension and the initial data, by means of a blow-up argument. These conditions were further shown to be sharp in the space $\dot{H}^{-\sigma}$ with $\sigma\in[0,n/2)$ (see \cite{Chen2023}). Therefore, it is natural to expect that the decay character theory also works well in the $L^m$ and $\dot{H}^{-\sigma}$ settings and leads to the corresponding critical exponents. However, for more general function spaces, establishing the sharpness of these conditions is considerably more difficult due to the lack of results relating the decay characters to the regularity of the initial data.

The main objectives of this paper are as follows. First, we prove the global (in time) existence of mild solutions to the equation \eqref{Semilinear_Damped_Waves} by combining the linear estimates established in \cite{Armando2022,ADL2024} with the decay character theory. Second, we investigate the blow-up phenomenon for the problem \eqref{Semilinear_Damped_Waves} with $(u_0,u_1)\in \mathcal{Y}^q\times \mathcal{Y}^q$ equipping $q \in (0, n/2)$, which enables us to determine the corresponding critical exponent. Here, the pseudo-measure spaces $\mathcal{Y}^q$ are defined by (see, for instance, \cite{Bhattacharya2003})
$$ 
\mathcal{Y}^q\left(\mathbb{R}^n\right):=\bigg\{f \in S^{\prime}\left(\mathbb{R}^n\right): \widehat{f} \in L_{\text{loc}}^1\left(\mathbb{R}^n\right) 
\text { and } \|f\|_{\mathcal{Y}^q} :=  \sup _{\xi \in \mathbb{R}^n}\left\{|\xi|^q|\widehat{f}(\xi)|\right\}<\infty\bigg\}, 
$$
where $S^{\prime}\left(\mathbb{R}^n\right)$ is the space of tempered distributions. These spaces were originally introduced in Harmonic Analysis and later proved to be particularly useful in the study of dissipative equations. The quantity
$|\xi|^q|\widehat{f}(\xi)|$ measures the behavior of $\widehat{f}$ near the low-frequency region $\xi = 0$. Hence, the parameter $q$ quantifies the strength of the singularity (or decay) of the Fourier transform at low frequencies. When $q$ increases, stronger control near $\xi=0$ is imposed. One can see that $\mathcal{Y}^q$ is slightly more general than the Riesz potential space $\dot{H}_1^{q}$ since $\||\cdot|^{q}\widehat{f}\|_{L^{\infty}}\lesssim \|\mathcal{I}_{-q}f\|_{L^1}$. These space are different from the $L^m$ and $\dot{H}^{-\sigma}$ spaces. To be specific, under the additional assumption that the initial data belong to the space $\mathcal{Y}^q$, the new critical exponent proposed for \eqref{Semilinear_Damped_Waves} is
\begin{equation}\label{Critical_Exp}
p_{\mathrm{crit}}(n,q,\gamma):=1+\frac{2+\gamma}{n-q}, \qquad 1\le n\le4,\quad 0\le\gamma<q<\frac{n}{2}.
\end{equation}
This result demonstrates the effectiveness of the decay character approach on determining the critical exponents for the semilinear problems. Finally, we establish sharp estimates for lifespan of local solutions to \eqref{Semilinear_Damped_Waves}.
More specifially, assuming that the initial data belong to the pseudo-measure spaces $\mathcal{Y}^q$ with $q$ satisfying some conditions, we establish the sharpness of the new lifespan estimates for the solution, namely
$$
T_{\varepsilon} \sim
\begin{cases}
\varepsilon^{-\frac{2(p-1)}{2+\gamma-(p-1)(n-q)}} & \text{ if } 1<p<p_{\mathrm{crit}}(n, q, \gamma),\\
\infty &\text{ if } p \geq p_{\mathrm{crit}}(n, q, \gamma).
\end{cases}
$$

To the best of the authors' knowledge, apart from initial data belonging to $L^1$, only initial data in $\mathcal{Y}^q$ allow one to derive sharp lifespan estimates for local solutions in the subcritical case (see Remark \ref{remark_sharp} for details).

\vspace{0.3cm}
\noindent \textbf{Notations:}   We write $f\lesssim g$ when there exists a constant $C>0$ such that $f\le Cg$, and $f \sim g$ when $g\lesssim f\lesssim g$. As usual, $H^{a}_m$ and $\dot{H}^{a}_m$, with $m \in (1, \infty), a \in \mathbb{R}$, denote potential spaces based on $L^m$ spaces. The space $\mathcal{Y}^q(\mathbb{R}^n)$ will be denoted simply by $\mathcal{Y}^q$. Finally, we denote $[\mu]^+
 := \max\{\mu, 0\}$ as the positive part of $\mu \in \mathbb{R}$, and 	$\lceil\mu \rceil
 := \min\{k \in \mathbb{Z} : k \geq \mu\}$.

\vspace{0.3cm}
\noindent \textbf{This article is organized as follows:} In Section 2, we establish the global existence theorem by means of the decay character theory and derive the corresponding decay estimates for solutions in the energy space. In Section 3, we prove that the conditions imposed on $p$ and $\gamma$ are sharp by means of a blow-up argument. In addition, we also establish sharp lifespan estimates for local solutions.

\section{Global existence of small data solution}\label{Section_GESDS}

\subsection{Linear estimates via decay characters}\label{Section_Linear}
We begin with a brief introduction to the decay character theory. First, we recall the notion of the decay indicator, which describes the behavior of the Fourier transform $(\widehat{|D|^s h})(\xi)$ near the origin in comparison with the power function $|\xi|^r$ (see \cite{Bjorland2009}). 

\begin{definition}[Decay indicator] \label{definition_2.1} For $h \in L^2\left(\mathbb{R}^n\right), s>0$ and $r \in(-n / 2+s, \infty)$, the decay indicator $P_r^s(h)$ corresponding to $|D|^s h$ is defined by
$$
P_r^s(h)=\lim _{\rho \rightarrow 0^+} \rho^{-2 r-n} \int_{B(\rho)}|\xi|^{2 s}\big|\widehat{h}(\xi)\big|^2 d \xi,
$$
where $B(\rho)$ is the ball centered at origin with radius $\rho$.
\end{definition} 
\begin{definition}[Decay character] \label{definition_2.2} The decay character of $|D|^s h$, denoted by $r_s^*\left(h\right)$ is the unique $r \in(-n / 2+s, \infty)$ such that $0<P_r^s\left(h\right)<\infty$, if this number $P_r^s\left(h\right)$ exists. If such $P_r^s\left(h\right)$ does not exist, we set $r_s^*(h)=-n / 2+s$ when $P_r^s\left(h\right)=+\infty$ for all $r \in(-n / 2+s, \infty)$, or $r_s^*(h)=+\infty$ when $P_r^s(h)=0$ for all $r \in(-n / 2+s, \infty)$.
For $s=0$, we denote $P_r^0\left(h\right)=P_r\left(h\right)$ and $r_0^*(h)=r^*(h)$.
\end{definition}
\noindent We denote by $\|\cdot\|_{P,s}$ the following norm of $h$ in $H^s$ for $s\geq 0$ (see \cite{ADL2024}):
\[
\|h\|_{P,s}:=\|h\|_{\dot{H}^s}+\|h\|_{L^2}+\big(P_{r^*(h)}(h)\big)^{1/2}.
\]
For convenience in calculations, we also denote the quantity
\begin{align*}
    \mathcal{R}(u_0, u_1) := \min\{r^*(u_0), r^*(u_1)\}.
\end{align*}
Now let us recall some estimates of solution to the linear problem \eqref{Linear_Damped_Waves}. 
\begin{proposition}[Theorem 1.1, \cite{ADL2024}] \label{Decay_Linear_Positive order}
Let $n \geq 1$ and $(v_0,v_1)\in H^1\times L^2$. If $-n/2<\mathcal{R}(v_0,v_1)<\infty$, then the solution $v$ to \eqref{Linear_Damped_Waves} satisfying the following decay estimates for $0\leq \alpha\leq 1$:
\begin{align*}
\|v(t,\cdot)\|_{\dot{H}^\alpha} \lesssim &(1+t)^{-\frac{n+2\alpha}{4}-\frac{\mathcal{R}(v_0, v_1)}{2}} (\|v_0\|_{P,\alpha} +\|v_1\|_{P,0}),\\
\|v_t(t,\cdot)\|_{L^2} \lesssim &(1+t)^{-\frac{n}{4}-\frac{\mathcal{R}(v_0, v_1)}{2}-1}(\|v_0\|_{P,1} +\|v_1\|_{P,0}).
\end{align*}
\end{proposition}
\subsection{Global existence result}

Based on the estimates established in Proposition \ref{Decay_Linear_Positive order}, we are able to prove the following theorem.
\begin{theorem}[\textbf{Global existence}] \label{Thm_Global existence}
Let $1 \leq n \leq 4$. The initial data $\left(u_0, u_1\right) \in \mathcal{D} := H^1\times L^2$ and the parameter $\gamma$ satisfy
\begin{equation}\label{Condition_R(u_0,u_1)}
0 \leq \gamma < -\mathcal{R}(u_0, u_1) < \frac{n}{2}.
\end{equation}
 We assume that $p$ satisfies the following conditions:
\begin{equation} \label{Condition_p>p crit}
p \geq 1+ \frac{2+\gamma}{n + \mathcal{R}(u_0, u_1)}
\end{equation}
and 
\begin{equation} \label{Condition_p for G-N}
\frac{2(n+\mathcal{R}(u_0, u_1)+\gamma)}{n} < p \leq \frac{n+2\gamma}{n-2}\;\text{ if }n=3,4. 
\end{equation}
Then there exists a constant $\varepsilon_0>0$ such that for all $\varepsilon \in (0,\varepsilon_0]$, the problem \eqref{Semilinear_Damped_Waves} admits a unique global (in time) solution $u \in \mathcal{C}\left([0, \infty), H^1\right)$. Furthermore, the following estimates hold:
\begin{align}
\|u(t,\cdot)\|_{L^2}&\lesssim \varepsilon(1+t)^{-\frac{n}{4}-\frac{\mathcal{R}(u_0, u_1)}{2}}\left\|\left(u_0, u_1\right)\right\|_{\mathcal{D}},\label{decay_estimate1}\\
\|u(t,\cdot)\|_{\dot{H}^1}&\lesssim \varepsilon(1+t)^{-\frac{n}{4}-\frac{\mathcal{R}(u_0, u_1)}{2}-\frac{1}{2}}\left\|\left(u_0, u_1\right)\right\|_{\mathcal{D}}\label{decay_estimate2}
\end{align}
where
$
    \left\|\left(u_0, u_1\right)\right\|_{\mathcal{D}}:=\left\|u_0\right\|_{P, 1}+\left\|u_1\right\|_{P, 0}.
$
\end{theorem}
\begin{remark}\label{Rem_Restrict Condition_Thm Global existence}
\fontshape{n}
\selectfont
When $n =3,4$, the appearance of the condition \eqref{Condition_p for G-N} comes from the use of Proposition \ref{fractionalgagliardonirenbergineq}. One can verify that the range of $p$ is not empty in this case. When $n=1,2$, it is obvious that $$1+\frac{2+\gamma}{n+\mathcal{R}(u_0, u_1)} > \frac{2(n+\mathcal{R}(u_0,u_1)+\gamma)}{n}.$$  Therefore, the left-hand side of condition (\ref{Condition_p for G-N}) is automatically satisfied when $n=1,2$. On the other hand, we require $\gamma<-\mathcal{R}(u_0, u_1)$ for the use of Proposition \ref{Hardy-Littlewood-Sobolev}. To relax those conditions, one may investigate Sobolev solutions by assuming that the initial data belong to Sobolev spaces with suitably higher, or even substantial, regularity; see \cite{Palmieri2018} for related discussions. However, in such a framework, the use of the fractional chain rule or fractional power estimates in handling the nonlinear terms inevitably imposes additional lower bound constraints on the exponent $p$.
\end{remark}

Next, we present some consequences derived from Theorem \ref{Thm_Global existence}.

\begin{remark}\label{globalexistence_Lm}
\fontshape{n}
\selectfont
If $(u_0,u_1)\in (L^m\cap L^2)^2$ with $m \in (1, 2]$, then $\mathcal{R}(u_0, u_1)\geq -n(1-1/m)$ (see \cite{ADL2024}). Moreover, this lower bound is optimal in the sense that one can choose $(u_0,u_1)$ satisfying $ \mathcal{R}(u_0, u_1) = r^*(u_0)=r^*(u_1)=-n\left(1-1/m\right)$ (see \cite{ADD2026}). Consequently, the condition \eqref{Condition_p>p crit} reduces to
\begin{equation}\label{p_crit_in_L^m}
p\ge 1+\frac{m(2+\gamma)}{n}
\end{equation}
and the condition (\ref{Condition_p for G-N}) is replaced by
\begin{align*}
    \frac{2}{m}+\frac{2\gamma}{n} < p \leq \frac{n+2\gamma}{n-2}. 
\end{align*}
The case $m=1$ is not considered, since $\mathcal{R}(u_0,u_1)\geq 0$ always holds (see \cite{ADL2024}), which immediately contradicts condition \eqref{Condition_R(u_0,u_1)}. 

A similar argument applies to the $\dot{H}^{-\sigma}$ framework, where $\sigma \in [0, n/2)$. Indeed, if $(u_0,u_1)\in (\dot{H}^{-\sigma}\cap L^2)^2$, then $\mathcal{R}(u_0, u_1) \geq \sigma-n/2$. Since $L^\eta$ is continuously embedded into $\dot{H}^{-\sigma}$ with $\eta=2n/(n+2\sigma)$, the above lower bound is again attainable, namely, $\mathcal{R}(u_0, u_1) = r^*(u_0)=r^*(u_1)=\sigma-n/2$. Therefore, condition \eqref{Condition_p>p crit} becomes
\begin{equation}\label{p_crit_in_H^sigma}
p\ge 1+\frac{4+2\gamma}{n+2\sigma}
\end{equation}
and the condition (\ref{Condition_p for G-N}) becomes
\begin{align*}
    1+\frac{2(\gamma+\sigma)}{n} < p \leq \frac{n+2\gamma}{n-2}. 
\end{align*}
When $\gamma=0$, estimates \eqref{decay_estimate1} and \eqref{decay_estimate2} recover the known decay rates of solutions, while \eqref{p_crit_in_L^m} and \eqref{p_crit_in_H^sigma} coincide with the well-known critical exponents in the $L^m$ ($m\in(1,2]$) and $\dot{H}^{-\sigma}$ ($\sigma\in(0,n/2)$) frameworks, respectively (see \cite{DAbbicco2017,Chen2023}). When $\gamma>0$, we show that they are still the critical exponents (see Remark \ref{blow_up_in_Lm}).
\end{remark}
\begin{remark} \label{Remark_Y^q}
\fontshape{n}
\selectfont
If $(u_0,u_1) \in (\mathcal{Y}^{q}\cap L^2)^2$ with $q\in (0, n/2)$, then $\mathcal{R}(u_0, u_1)\geq -q$ (see \cite{ADD2026}). Moreover, we can choose
\[
u_0(x)=u_1(x)=\mathfrak{F}^{-1}(v(\xi))(x),
\]
where
$
v(\xi)\sim |\xi|^{-q} \text{ for }|\xi|\le 1 \text{ and } 
v(\xi)\sim |\xi|^{-n/2-\varepsilon}\quad\text{for }|\xi|\ge 2.
$
We see that $u_0,u_1\in L^2\cap\mathcal{Y}^{q}$ and $|\xi|^{q}\widehat{u_0}(\xi)\sim |\xi|^{q}\widehat{u_1}(\xi)\sim1$ as $|\xi|\to 0$. Consequently, a direct calculation yields
$
\mathcal{R}(u_0, u_1) = r^*(u_0)=r^*(u_1)=-q.
$
Hence, the condition \eqref{Condition_p>p crit} becomes
\begin{equation}\label{p_crit_Gamma_q}
p \geq p_{\mathrm{crit}}(n,q,\gamma):=1+\frac{2+\gamma}{n-q}
\end{equation}
and the condition \eqref{Condition_p for G-N} becomes
\begin{equation}\label{G-N_p_crit_Gamma_q}
\displaystyle\frac{2(n-q+\gamma)}{n}<p \leq \displaystyle\frac{n + 2\gamma}{n-2}.
\end{equation}
We will show that \eqref{p_crit_Gamma_q} is the critical exponent for problem \eqref{Semilinear_Damped_Waves} by proving blow-up result for the  suitable initial data in $\mathcal{Y}^{q}$ with $q\in (0, n/2)$ (see Section \ref{Section_Blow-up}).
\end{remark}

\begin{proof}[\textbf{Proof of Theorem \ref{Thm_Global existence}.}]

To begin with, we can write the solution to the linear problem \eqref{Linear_Damped_Waves} by the formula
$$
u^{\rm lin}(t,x) = \varepsilon (\mathcal{K}(t,x) +\partial_t \mathcal{K}(t,x))\ast_x u_0(x) + \varepsilon \mathcal{K}(t,x) \ast_x u_1(x),
$$
so that the solution to \eqref{Semilinear_Damped_Waves} becomes $$u(t,x)= u^{\rm lin}(t,x) + u^{\rm non}(t,x),$$ thanks to Duhamel's principle, where $$u^{\rm non}(t,x) := \int_0^t \mathcal{K}(t-\tau,x) \ast_x \mathcal{I}_\gamma(|u|^p)(\tau,x) d\tau$$ and the Fourier transform of the linear kernel $\mathcal{K}(t,x)$ defined by
\begin{align*}
\widehat{\mathcal{K}}(t,\xi)= \begin{cases}
\displaystyle\frac{e^{-\frac{t}{2}}\sinh{\left(t \sqrt{\frac{1}{4} -|\xi|^2}\right)}}{\sqrt{\frac{1}{4}- |\xi|^2}}  &\text{ if } |\xi| \leq \displaystyle\frac{1}{2},\\
\displaystyle\frac{e^{-\frac{t}{2}}\sin{\left(t \sqrt{|\xi|^2-\frac{1}{4}}\right)}}{\sqrt{|\xi|^2-\frac{1}{4}}} &\text{ if } |\xi| > \displaystyle\frac{1}{2}.
\end{cases}
\end{align*}

Under the assumptions of Theorem \ref{Thm_Global existence}, we introduce the following solution space for $T > 0$: $$X(T):=\mathcal{C}([0,T],H^1),$$ carrying its norm
\begin{align*}
\|\varphi\|_{X(T)}:=\sup\limits_{t\in[0,T]}\left\{(1+t)^{\frac{n}{4}+\frac{\mathcal{R}(u_0, u_1)}{2}}\|\varphi(t,\cdot)\|_{L^2}  +(1+t)^{\frac{n}{4}+\frac{\mathcal{R}(u_0, u_1)}{2}+\frac{1}{2}}\|\varphi(t,\cdot)\|_{\dot{H}^1}\right\}
\end{align*}
and the closed ball
     \begin{align*}
         X(T, K) := \left\{\varphi \in X(T) : \|\varphi\|_{X(T)} \leq K\right\}
     \end{align*}
     for $K > 0$.
In addition, we define the operator $\mathcal{N}$ on the space $X(T)$ as follows:
\begin{align*}
\mathcal{N}:\ u(t,x)\in X(T)\to \mathcal{N}[u](t,x):=u^{\text{lin}} (t,x)+u^{\text{non}}(t,x).
\end{align*}
In order to prove Theorem \ref{Thm_Global existence}, we will show that the following two crucial inequalities hold for all $u, \bar{u} \in X(T)$:
\begin{align} \label{Cruc-01}
\|\mathcal{N}[u]\|_{X(T)}\lesssim & \,\varepsilon\left\|\left(u_0, u_1\right)\right\|_{\mathcal{D}}+\|u\|_{X(T)}^p,\\
\label{Cruc-02} \|\mathcal{N}[u]-\mathcal{N}[\bar{u}]\|_{X(T)}\lesssim &\,\|u-\bar{u}\|_{X(T)}\left(\|u\|_{X(T)}^{p-1}+\|\bar{u}\|_{X(T)}^{p-1}\right).
\end{align}
Then, if we choose $M$ and $\varepsilon_0 > 0$ such that
     \begin{align*}
         M:=  2C_1\|(u_0, u_1)\|_{\mathcal{D}} \quad\text{ and }\quad \max\{C_1, 2C_2\} M^{p-1} \varepsilon_0^{p-1} < \frac{1}{2},
     \end{align*}
      then $\mathcal{N}$ is a contraction mapping on $X(T, M\varepsilon)$ for $\varepsilon \in (0, \varepsilon_0]$. Therefore, by the Banach fixed point theorem, we obtain a unique solution $u^* = \mathcal{N}[u^*] \in X(T, M\varepsilon)$ for all $T > 0$. Finally, since $T$ is arbitrary, we conclude that $u^* \in X(\infty, M\varepsilon)$.
 
 Now, we consider the following two cases:

\begin{itemize}
[leftmargin=*]
    \item \noindent \textbf{Case 1: The supercritical case $p> 1 + (2+\gamma)/(n+\mathcal{R}(u_0, u_1)) $}.

First, we verify the inequality \eqref{Cruc-01}. Using Proposition \ref{Decay_Linear_Positive order}, we immediately obtain $$\|u^{\text{lin}}\|_{X(T)}\lesssim\varepsilon\left\|\left(u_0, u_1\right)\right\|_{\mathcal{D}}.$$ Consequently, we aim to prove
\begin{equation}\label{Crucial_03}
\|u^{\text{non}}\|_{X(T)}\lesssim\|u\|_{X(T)}^p.
\end{equation}
We can see that the condition (\ref{Condition_R(u_0,u_1)}) leads to
\begin{align*}
    \frac{2n}{n+2\gamma} > 1 \quad\text{ and }\quad \eta := \frac{n}{n+\mathcal{R}(u_0, u_1)} \in \left(\frac{n}{n-\gamma},2\right).
\end{align*}
For $k = 0,1$, using the estimate Lemma \ref{lemma:1.3}, we get
$$
\|u^{\rm non}(t,\cdot)\|_{\dot{H}^k} \lesssim \int_0^t (1+t-\tau)^{-\frac{n}{2}(\frac{1}{\eta}-\frac{1}{2})-\frac{k}{2}} \|\mathcal{I}_\gamma(|u|^p)(\tau,\cdot)\|_{L^{\eta} \cap L^2} d\tau.
$$
 Applying Propositions \ref{fractionalgagliardonirenbergineq}-\ref{Hardy-Littlewood-Sobolev} and the norm definition of $X(T)$, we arrive at
\begin{align*}
\|\mathcal{I}_\gamma(|u|^p)(\tau,\cdot)\|_{L^2}\lesssim &\|u(\tau, \cdot)\|_{L^{\frac{2np}{n+2\gamma}}}^p \lesssim \|u(\tau,\cdot)\|^{p(1-\theta_1)}_{L^2}\|u(\tau,\cdot)\|^{p\theta_1}_{\dot{H}^1}\\
\lesssim & (1+\tau)^{-\frac{n+\mathcal{R}(u_0, u_1)}{2} p+\frac{n+2\gamma}{4}}\|u\|_{X(T)}^p, \\ 
\| \mathcal{I}_\gamma(|u|^p)(\tau,\cdot)\|_{L^{\eta}}\lesssim & \|u(\tau, \cdot)\|_{L^{\frac{\eta np}{n+\eta\gamma}}}^p \lesssim \|u(\tau,\cdot)\|^{p(1-\theta_2)}_{L^2}\|u(\tau,\cdot)\|^{p\theta_2}_{\dot{H}^1}\\
\lesssim & (1+\tau)^{-\frac{n+\mathcal{R}(u_0, u_1)}{2}p+\frac{n}{2\eta}+\frac{\gamma}{2}}\|u\|_{X(T)}^p, 
\end{align*}
where we have $2n/(n+2\gamma) > 1$, $\eta n/(n+\eta \gamma) > 1$ and  the condition \eqref{Condition_p for G-N} implies that
$$
\theta_1:= \frac{n}{2}\left(1-\frac{1}{p}\right)-\frac{\gamma}{p} \in [0,1] \text{ and } \theta_2:=n\left(\frac{1}{2}-\frac{1}{p}\right)-\frac{\gamma+\mathcal{R}(u_0, u_1)}{p} \in [0,1].
$$
Therefore, we obtain the following estimate:
\begin{align*}
&\|u^{\text{non}}(t,\cdot)\|_{\dot{H}^k} 
\lesssim  \|u\|^p_{X(T)} \int_0^t(1+t-\tau)^{-\frac{n}{2}(\frac{1}{\eta}-\frac{1}{2})-\frac{k}{2}}(1+\tau)^{-\frac{n+\mathcal{R}(u_0, u_1)}{2}p+\frac{n}{2\eta}+\frac{\gamma}{2}}d\tau.
\end{align*}
The condition \eqref{Condition_p>p crit} leads to $$\frac{n+\mathcal{R}(u_0, u_1)}{2}p-\frac{n}{2\eta}-\frac{\gamma}{2} > 1.$$ Moreover, for $1 \leq n \leq 4$, we observe that 
$$
\frac{n}{2}\left(\frac{1}{\eta}-\frac{1}{2}\right)+\frac{k}{2}\leq \frac{n+\mathcal{R}(u_0, u_1)}{2}p-\frac{n}{2\eta}-\frac{\gamma}{2}.
$$
Hence, employing Lemma \ref{prop_Integral inequality},  we arrive at 
\begin{align*}
&\int_0^t(1+t-\tau)^{-\frac{n}{2}(\frac{1}{\eta}-\frac{1}{2})-\frac{k}{2}}(1+\tau)^{-\frac{n+\mathcal{R}(u_0, u_1)}{2}p+\frac{n}{2\eta}+\frac{\gamma}{2}}d\tau \\
\lesssim & (1+t)^{-\frac{n}{2}(\frac{1}{\eta}-\frac{1}{2})-\frac{k}{2}} = (1+t)^{-\frac{n}{4}-\frac{\mathcal{R}(u_0, u_1)}{2}-\frac{k}{2}}.
\end{align*}
Summarizing, we conclude that
\begin{equation*}
(1+t)^{\frac{n}{4}+\frac{\mathcal{R}(u_0, u_1)}{2}+\frac{k}{2}} \|u^{\rm non}(t,\cdot)\|_{\dot{H}^k} \lesssim \|u\|_{X(T)}^p
\end{equation*}
for $k=0,1$.
This immediately implies \eqref{Crucial_03}.

\noindent In order to prove \eqref{Cruc-02}, we notice that the following relation for all $u, \bar{u} \in X(T)$:
$$
\|\mathcal{N}[u]-\mathcal{N}[\bar{u}]\|_{X(T)}=\left\|\int_0^t \mathcal{K}(t-\tau,x)\ast_{x}\left(\mathcal{I}_\gamma(|u|^p)(\tau,x)-\mathcal{I}_\gamma(|\bar{u}|^p)(\tau,x)\right)d\tau \right\|_{X(T)}.
$$
Applying Lemma \ref{Hardy-Littlewood-Sobolev}, we achieve
\begin{align*}
\left\| \mathcal{I}_\gamma(|u|^{p}-|\bar{u}|^{p})(\tau, \cdot) \right\|_{L^2}\lesssim & \||u(\tau, \cdot)|^p-|\bar{u}(\tau, \cdot)|^p\|_{L^{\frac{2n}{n+2\gamma}}}, \\
\left\| \mathcal{I}_\gamma(|u|^{p}-|\bar{u}|^{p})(\tau, \cdot) \right\|_{L^{\eta}} \lesssim & \||u(\tau, \cdot)|^p-|\bar{u}(\tau, \cdot)|^p\|_{L^{\frac{\eta n}{n+\eta\gamma}}}.
\end{align*}
On the other hand, since
$$
||u(\tau, \cdot)|^p-|\bar{u}(\tau, \cdot)|^p|
\lesssim \left|u(\tau, \cdot)-\bar{u}(\tau, \cdot)\right|\left(\left|u(\tau, \cdot)\right|^{p-1}+\left|\bar{u}(\tau, \cdot)\right|^{p-1}\right),
$$
employing H\"{o}lder's inequality one finds
\begin{align*}
\||u(\tau, \cdot)|^p-|\bar{u}(\tau, \cdot)|^p\|_{L^{\frac{2n}{n+2\gamma}}} 
\lesssim & \left\|u(\tau, \cdot)-\bar{u}(\tau, \cdot)\right\|_{L^{\frac{2np}{n+2\gamma}}} \left(\left\|u(\tau, \cdot)\right\|_{L^{\frac{2np}{n+2\gamma}}}^{p-1}+\left\|\bar{u}(\tau, \cdot)\right\|_{L^{\frac{2np}{n+2\gamma}}}^{p-1}\right), \\
\||u(\tau, \cdot)|^p-|\bar{u}(\tau, \cdot)|^p\|_{L^{\frac{\eta n}{n+\eta\gamma}}} 
\lesssim & \left\|u(\tau, \cdot)-\bar{u}(\tau, \cdot)\right\|_{L^{\frac{\eta np}{n+\eta\gamma}}} \left(\left\|u(\tau, \cdot)\right\|_{L^{\frac{\eta np}{n+\eta\gamma}}}^{p-1}+\left\|\bar{u}(\tau, \cdot)\right\|_{L^{\frac{\eta np}{n+\eta\gamma}}}^{p-1}\right) .
\end{align*}
Finally, employing Proposition \ref{fractionalgagliardonirenbergineq} to estimate the terms $\left\|u(\tau, \cdot)-\bar{u}(\tau, \cdot)\right\|_{L^\omega}, \left\|u(\tau, \cdot)\right\|_{L^\omega}$ and $\left\|\bar{u}(\tau, \cdot)\right\|_{L^\omega}$, appearing in the previous inequality, where $\omega=(2np)/(n+2\gamma)$ or $\omega=(\eta np)/(n+\eta\gamma)$, we obtain the desired estimate \eqref{Cruc-02}.

\item \noindent \textbf{Case 2:} The critical case $p=1 + (2+\gamma)/(n+\mathcal{R}(u_0, u_1))$.

The conditions
$
p> 2(n+\mathcal{R}(u_0, u_1)+\gamma)/ n = 2(n+\eta\gamma)/(\eta n)$ and $ \gamma<-\mathcal{R}(u_0, u_1)
$
imply that there exists $\kappa<\eta$ satisfies $$\kappa\in \left(\frac{n}{n-\gamma},\frac{n}{n+\mathcal{R}(u_0, u_1)}\right) \quad\text{ and }\quad p\geq \frac{2(n+\kappa \gamma)}{\kappa n}>\frac{2(n+\eta \gamma)}{\eta n}.$$  We can replace $\eta $ by $\kappa$ in the supercritical case. Propositions \ref{fractionalgagliardonirenbergineq}-\ref{Hardy-Littlewood-Sobolev} are still valid and we have
$$
\|\mathcal{I}_{\gamma}(|u|^{p})(\tau,\cdot)\|_{L^{\kappa}}\lesssim \|u(\tau,\cdot)\|^p_{L^{\frac{\kappa np}{n+\kappa\gamma}}}\lesssim (1+\tau)^{-\frac{n+\mathcal{R}(u_0, u_1)}{2}p+\frac{n}{2\kappa}+\frac{\gamma}{2}}\|u\|^{p}_{X(T)}.
$$
For $k = 0,1$, applying Lemma \ref{prop_Integral inequality}, we deduce that 
\begin{align*}
&\int_{0}^{t}(1+t-\tau)^{-\frac{n}{2}(\frac{1}{\kappa}-\frac{1}{2})-\frac{k}{2}}(1+\tau)^{-\frac{n+\mathcal{R}(u_0, u_1)}{2}p+\frac{n}{2\kappa}+\frac{\gamma}{2}}d\tau \\
\lesssim&\int_{0}^{t}(1+t-\tau)^{-\frac{n}{2}(\frac{1}{\eta}-\frac{1}{2})-\frac{k}{2}}(1+\tau)^{-\frac{n+\mathcal{R}(u_0, u_1)}{2}p+\frac{n}{2\eta}+\frac{\gamma}{2}+\frac{n}{2}(\frac{1}{\kappa}-\frac{1}{\eta})}d\tau\\
=&\int_{0}^{t}(1+t-\tau)^{-\frac{n}{2}(\frac{1}{\eta}-\frac{1}{2})-\frac{k}{2}}(1+\tau)^{-1+\frac{n}{2}(\frac{1}{\kappa}-\frac{1}{\eta})}d\tau\\
\lesssim & (1+t)^{-\frac{n}{2}(\frac{1}{\eta}-\frac{1}{2})-\frac{k}{2}}=(1+t)^{-\frac{n+2k}{4}-\frac{\mathcal{R}(u_0, u_1)}{2}}.
\end{align*}
Repeat some steps in the supercritical case, and we obtain the same result. 
\end{itemize}
Hence, the proof of Theorem \ref{Thm_Global existence} is completed.
\end{proof}

\section{Blow-up results and sharp lifespan estimates for solutions with initial data from pseudo-measure spaces}\label{Section_Blow-up}

\subsection{Blow-up results}
In this section, we will prove that $$p_{\text{crit}}(n,q,\gamma):=1+\frac{2+\gamma}{n-q}$$ is really the critical exponent of the problem \eqref{Semilinear_Damped_Waves} with $(u_0,u_1)\in \mathcal{Y}^q\times \mathcal{Y}^q$ equipping $q \in (0, n/2)$. To achieve this aim, we give a definition of local/global weak solution to \eqref{Semilinear_Damped_Waves}.

\begin{definition}[\textbf{Weak solution}] \label{weak_def}
Let $p > 1$ and $T \in (0, \infty)$. We say that $u \in \mathcal{C}([0,T), L^2)$ is a local (in time) weak solution to \eqref{Semilinear_Damped_Waves} if  for any test functions $\chi \in \mathcal{C}_0^\infty([0,T))$ and $\phi \in \mathcal{C}_0^{\infty}(\mathbb{R}^n)$,  the following relation holds: 
\begin{align}
\int_0^{T}\int_{ \mathbb{R}^n}\mathcal{I}_{\gamma} (|u|^p)(t,x)\phi(x)\chi(t)dxdt = \int_0^{T}\int_{ \mathbb{R}^n}(\partial_t^2-\Delta+\partial_t)u(t,x)\phi(x)\chi(t)dxdt. \label{eqweak}
\end{align}
If $T = \infty$, then we say that $u$ is a global (in time) weak solution to \eqref{Semilinear_Damped_Waves}.
\end{definition}
We now proceed to examine the blow-up phenomenon in the subcritical regime.
\begin{theorem}[\textbf{Blow-up}] \label{Thm_Blow_up}
Let $n \geq 1$, $q \in (0, n/2)$ and $\gamma\in [0,n)$. The exponent $p$ fulfills
\begin{align}
1<p<p_{\mathrm{crit}}(n, q, \gamma). \label{condition-Blowup}
\end{align}
Let us assume $\left(u_0,u_1\right)\in \mathcal{Y}^q\times\mathcal{Y}^q$ such that
\begin{equation} \label{assumption_initial_data}
u_0(x)+u_1(x)\gtrsim\langle x\rangle^{-n+q},
\end{equation}
where we denote $\langle x\rangle:= \left(1+|x|^2\right)^{1/2}$ by the Japanese bracket for $x \in \mathbb{R}^n$. Then, there is no global (in time) weak solution by mean of Definition \ref{weak_def} to \eqref{Semilinear_Damped_Waves}.
\end{theorem}
The non-emptiness of the data set in Theorem \ref{Thm_Blow_up} will be discussed in Lemma \ref{non-empty}.
\begin{remark}
\fontshape{n}
\selectfont
According to Remark \ref{Remark_Y^q} and Theorem \ref{Thm_Blow_up}, we conclude that the critical exponent for the problem \eqref{Semilinear_Damped_Waves} with initial data belonging additionally to $\mathcal{Y}^q$ is defined by  \eqref{Critical_Exp}. It provides a new viewpoint for the critical exponent of semilinear damped wave equations. 
\end{remark}
\begin{example}
\fontshape{n}
\selectfont
    We give some examples where the sharpness of $p$ can be confirmed. 
    \begin{itemize}[leftmargin=*]
        \item If $n=1,2$ and $0\leq \gamma<q<n/2$: We obtain the critical exponent $p$ on the entire interval $(1,\infty)$.
        \item If $n=3,4$ and $q\in (0,n/2)$ is fixed: The critical exponent $p$ on the interval $(1,(n+2\gamma)/(n-2))$ can be claimed if $0\leq \gamma<\min\left\{q,2n/(n-2q)-n+q\right\}$.
        \item If $n=3,4$ and $\gamma\in (0,n/2)$ is fixed: Let $q_1\leq q_2$ be the solutions of $$2q^2-(3n+2\gamma)q+n^2-2n-n\gamma=0.$$ The critical exponent $p$ on the interval $(1,(n+2\gamma)/(n-2))$ can be claimed if $q\in [q_1,q_2]$.
    \end{itemize}
\end{example}
\begin{proof}[\textbf{Proof of Theorem \ref{Thm_Blow_up}.}]
At first, let us introduce the test functions $\chi= \chi(t)$ and $\phi=\phi(x)$ satisfying the following properties:
\begin{align*}
&1.\quad \chi \in \mathcal{C}_0^\infty([0,\infty)) \text{ and }
\chi(t)=
\begin{cases}
1 &\text{ if } 0 \leq t \leq 1/2, \\
\text{decreasing} &\text{ if } 1/2\leq t\leq 1, \\
0 &\text{ if } t \geq 1,
\end{cases} & \nonumber \\
&2.\quad \phi \in \mathcal{C}_0^\infty(\mathbb{R}^n) \text{ and }
\phi(x)=
\begin{cases}
1 &\text{ if }0 \leq |x| \leq 1/2, \\
\text{decreasing} &\text{ if } 1/2\leq |x| \leq 1, \\
0 &\text{ if } |x| \geq 1,
\end{cases} & \nonumber \\
&3.\quad \chi^{-\frac{p'}{p}}(t)\big(|\chi'(t)|^{p'}+|\chi''(t)|^{p'}\big) \text{ and } \phi^{-\frac{p'}{p}}(x)\left|\Delta \phi(x)\right|^{p'} \text{ are bounded }.
\end{align*}
where $p'$ stands for the conjugate of $p$. For $R>0$, we define $\phi_R(x):=\phi(R^{-1}x)$ and $\chi_R(t):=\chi(R^{-2}t)$. Assume by contradiction that $u=u(t,x)$ is a global (in time) weak solution to \eqref{Semilinear_Damped_Waves}.
 We apply Definition \ref{weak_def} with $\chi$ and $\phi$ replaced by $\chi_R$ and $\phi_R$, respectively.
From this, by integration by parts, the relation \eqref{eqweak} implies that
\begin{align}\label{eqintegral}
&-\varepsilon\int_{\mathbb{R}^n}(u_0(x)+u_1(x))\phi_R(x)dx+\int_0^{\infty}\int_{\mathbb{R}^n}u(t, x)\phi_R(x)\partial_t^2\chi_R(t) dxdt \notag\\
&-\int_0^{\infty}\int_{\mathbb{R}^n} u(t, x)\phi_R(x)\partial_t\chi_R(t)dxdt-\int_0^{\infty}\int_{\mathbb{R}^n}u(t, x)\Delta  \phi_R(x)\chi_R(t)dxdt \notag\\ =&\int_0^{\infty}\int_{\mathbb{R}^n}\mathcal{I}_{\gamma} \left(|u|^p\right)(t,x)\phi_R(x)\chi_{R}(t)dxdt.
\end{align}
We define
$$
\mathcal{J}(R):=\int_0^{\infty}\int_{\mathbb{R}^n}|u(t, x)|^p \phi_R(x)\chi_R(t)dxdt=\int_0^{R^2}\int_{0 \leq|x| \leq R}|u(t, x)|^p \phi_R(x)\chi_R(t)dxdt.
$$
We deduce from \eqref{eqintegral} the inequality
\begin{align}
&\int_0^{R^2}\int_{0 \leq |x| \leq R} \mathcal{I}_{\gamma} \left(|u|^p\right)(t,x)\phi_R(x)\chi_R(t)dxdt+\varepsilon\int_{0 \leq |x| \leq R}(u_0(x)+u_1(x))\phi_R(x)dx \notag\\
\leq & \int_{R^2/2}^{R^2}\int_{0 \leq |x| \leq R} |u(t, x)| |\phi_R(x) \partial_t\chi_R(t)|dxdt+\int_{R^2/2}^{R^2}\int_{0 \leq |x| \leq R} |u(t, x)||\phi_R(x) \partial_t^2\chi_R(t)|dxdt \notag\\
&+
\int_0^{R^2}\int_{R/2 \leq |x| \leq R}|u(t, x)||\Delta \phi_R(x)|\chi_R(t)dxdt \notag\\
=:&\,\,\mathcal{J}_{1,R} + \mathcal{J}_{2,R} + \mathcal{J}_{3,R}. \label{eqabs}
\end{align}
Using H\"older's inequality with $1/p + 1/p' = 1$ and the change of
variables $\overline{x}=R^{-1}x$, $\overline{t}=R^{-2}t$, we obtain
\begin{align}
\mathcal{J}_{1,R}
\lesssim & \mathcal{J}_R^{\frac{1}{p
}}\left(\int_{R^2/2}^{R^2}\int_{0 \leq |x| \leq R}\chi^{-\frac{p'}{p}}_R(t) |\partial_t\chi_R(t)|^{p'}\phi_R(x)dxdt\right)^{\frac{1}{p'}} \notag\\
\lesssim& \mathcal{J}_R^{\frac{1}{p}}R^{-2+\frac{n+2}{p'}} \left(\int_{1/2}^{1}\int_{0 \leq |\bar{x}| \leq 1} |\chi^{\prime}(\overline{t})|^{p'}|\chi(\overline{t})|^{-\frac{p^{\prime}}{p}}\phi(\overline{x}) d\overline{x}d\overline{t}\right)^{\frac{1}{p'}} \sim \mathcal{J}_R^{\frac{1}{p}}R^{-2+\frac{n+2}{p'}}.\label{eqHolder1}
\end{align}
The same argument yields
\begin{equation}\label{eqHolder2}
\mathcal{J}_{2,R} \lesssim \mathcal{J}_R^{\frac{1}{p}}R^{-4+\frac{n+2}{p'}}.
\end{equation}
For the term $\mathcal{J}_{3,R}$, we treat it as follows:
\begin{align}
\mathcal{J}_{3,R}
&\lesssim \mathcal{J}_R^{\frac{1}{p}}
\left(\int_0^{R^2}\int_{R/2 \leq |x| \leq R} |\Delta\phi_R(x)|^{p'} \phi_R^{-\frac{p'}{p}}(x)\chi_R(t)dxdt\right)^{\frac{1}{p'}} \lesssim \mathcal{J}_R^{\frac{1}{p}}R^{-2+\frac{n+2}{p'}}. \label{eqHolder3}
\end{align}
By \eqref{eqabs}, \eqref{eqHolder1}, \eqref{eqHolder2} and \eqref{eqHolder3}, it follows that
\begin{align} \label{I_u^p}
\int_0^{R^2}\int_{0 \leq |x| \leq R} \mathcal{I}_{\gamma} \left(|u|^p\right)(t,x)\phi_R(x)\chi_R(t)dxdt &+\varepsilon\int_{0 \leq |x| \leq R}(u_0(x)+u_1(x))\phi_{R}(x)dx\notag\\ &\quad\leq C \mathcal{J}_R^{\frac{1}{p}}R^{-2+\frac{n+2}{p'}}.
\end{align}
To control the first quantity on the left-hand side of \eqref{I_u^p}, we begin with the representation of the Riesz potential
\begin{align*}
\mathcal{I}_{\gamma}\left(|u|^p\right)(t,x)= C_\gamma\int_{\mathbb{R}^n} \displaystyle\frac{|u(t, y)|^p}{|x-y|^{n-\gamma}} d y
\geq & \,\,C_\gamma\int_{\mathcal{D}_R} \displaystyle\frac{|u(t, y)|^p}{|x-y|^{n-\gamma}} d y
\end{align*}
where $ \mathcal{D}_R=\left\{y \in \mathbb{R}^n, |y| \leq R\right\}$.
Now observe that for $x \in \overline{\mathcal{D}}_R=\left\{x \in \mathbb{R}^n, 0 \leq |x| \leq R/2\right\}$ and $y \in \mathcal{D}_R=\left\{y \in \mathbb{R}^n, 0 \leq |y| \leq R\right\}$, the triangle inequality gives $|x-y| \leq|x|+|y| \leq R/2+R \leq 2 R$. Consequently, $|x-y|^{n-\gamma} \leq 2^{n-\gamma} R^{n-\gamma}$ for all $x \in \overline{\mathcal{D}}_R,\,\, y \in \mathcal{D}_R$. This leads to the pointwise lower bound
$$
\mathcal{I}_{\gamma}\left(|u|^p\right)(t, x) \geq \frac{C_\gamma R^{-(n-\gamma)}}{2^{n-\gamma}} \int_{\mathcal{D}_R}|u(t, y)|^p d y \quad \text { for all } t \in (0, R^2), x \in \overline{\mathcal{D}}_R.
$$
Then, we get
\begin{align} \label{estimate_I_gamma}
&\int_0^{R^2} \int_{\overline{\mathcal{D}}_R} \mathcal{I}_{\gamma}\left(|u|^p\right)(t,x) \phi_R(x)\chi_R(t) d x d t \notag\\
\geq & \frac{R^{-(n-\gamma)}}{2^{n-\gamma}} \int_0^{R^2} \int_{\overline{\mathcal{D}}_R} \int_{\mathcal{D}_R} |u(t, y)|^p \phi_R(x)\chi_R(t) d y d x d t \notag\\
\geq & \operatorname{meas}(\overline{\mathcal{D}}_R) \frac{R^{-(n-\gamma)}}{2^{n-\gamma}} \int_0^{R^2} \int_{\mathcal{D}_R} |u(t, y)|^p \phi_R(y)\chi_R(t) d y d t \geq C R^\gamma \mathcal{J}_R,
\end{align}
where we have used the fact that $\phi_R \equiv 1$ on $\overline{\mathcal{D}}_R$, and $1 \geq \phi_R$ on $\mathcal{D}_R$. Moreover, a simple calculation leads to
\begin{equation} \label{estimate_initial_data}
\int_{0 \leq |x| \leq R}(u_0(x)+u_1(x))\phi_{R}(x)dx \geq C  \int_{0 \leq |x|\leq R}\langle x\rangle^{-n+q} \geq C  R^q.
\end{equation}
for $R\gg1$. Combining \eqref{I_u^p}, \eqref{estimate_I_gamma}, and \eqref{estimate_initial_data}, we infer that $$R^\gamma \mathcal{J}_R+ \varepsilon R^q \leq C \mathcal{J}_R^{\frac{1}{p}}R^{-2+\frac{n+2}{p'}}.$$ This leads to 
\begin{align*}
\varepsilon R^{q-\gamma} \leq C\mathcal{J}_R^{\frac{1}{p}}R^{-2-\gamma+\frac{n+2}{p'}} -\mathcal{J}_R \lesssim R^{-(2+\gamma)p'+n+2},
\end{align*}
that is 
\begin{align}
\varepsilon \lesssim R^{-(2+\gamma)p'+n+2+\gamma-q}. \label{eq1thr6.1}
\end{align}
The assumption \eqref{condition-Blowup} is equivalent to
\begin{align*}
-(2+\gamma)p'+n+2+\gamma-q < 0.
\end{align*}
We pass $R \to \infty$ in \eqref{eq1thr6.1} to derive a contradiction. This completes the proof of Theorem \ref{Thm_Blow_up}.
\end{proof}
\begin{remark}\label{blow_up_in_Lm}
{\rm
The proof of Theorem \ref{Thm_Blow_up} can be adapted to prove the following results, with a small modification on the initial data:
\begin{itemize}
    \item Let $u_0,u_1\in L^m$ ($m > 1$) such that 
    \[
    u_0(x)+u_1(x) \gtrsim \langle x\rangle^{-n/m}(\log(\mathrm{e}+|x|))^{-1}.
    \]
    Then problem \eqref{Semilinear_Damped_Waves} admits no global (in time) weak solution for $1<p<1+m(2+\gamma)/n$. Combining with Remark \ref{globalexistence_Lm}, it shows that $p_{\text {crit }}(n, \sigma)=1+m(2+\gamma)/n$ is critical for initial data from $L^m$ spaces.
    \item Let $u_0,u_1\in \dot{H}^{-\sigma}$ ($\sigma\in [0,n/2)$) satisfy 
    \[
    u_0(x)+u_1(x)\gtrsim \langle x\rangle^{-n\left(1/2+\sigma/n\right)}(\log (\mathrm{e}+|x|))^{-1}.
    \]
    Then problem \eqref{Semilinear_Damped_Waves} admits no global (in time) weak solution for $1<p<1+(4+2\gamma)/(n+2\sigma)$. Combining with Remark \ref{globalexistence_Lm}, it shows that $p_{\text {crit }}(n, \sigma)=1+(4+2\gamma)/(n+2\sigma)$ is critical for initial data from $\dot{H}^{-\sigma}$ spaces.
\end{itemize}

}
\end{remark}
\subsection{Sharp estimates for lifespan}

In this section, we will summarize how to get sharp estimates for the lifespan of solutions in the subcritical case  $1 < p < p_{\rm crit}(n,q, \gamma)$ when the initial data belong to the pseudo-measure $\mathcal{Y}^q$. Then, lower bound estimates and upper bound estimates
for the lifespan $T_\varepsilon$ are given by the next statements.

\begin{proposition}[\textbf{Upper bound of lifespan}]\label{upperbound}
Let $u_0,u_1\in \mathcal{Y}^q$ and all conditions in Theorem \ref{Thm_Blow_up} be satisfied. Then, there exists a constant $\varepsilon_1>0$ such that for all $\varepsilon \in\left(0, \varepsilon_1\right]$, the lifespan $T_{\varepsilon, w}$ of local (in time) weak solutions to \eqref{Semilinear_Damped_Waves} is valid:
\begin{align}
T_{\varepsilon, w}\lesssim \varepsilon^{-\frac{2(p-1)}{2+\gamma-(p-1)(n-q)}}. \label{Esti1}
\end{align}
\end{proposition}
\begin{proof}[\textbf{Proof of Proposition \ref{upperbound}.}]
We recall the relation \eqref{eq1thr6.1}. 
Let $R\to \sqrt{T_{\varepsilon, w}}$, we obtain the estimate \eqref{Esti1}. Our proof is finished.
\end{proof}

As established in Theorem \ref{Thm_Blow_up}, when $1<p<p_{\mathrm{crit}}(n, q, \gamma)$, a nontrivial weak solution, defined locally in time, may experience blow-up within a finite time interval. This phenomenon motivates us to investigate the lifespan of solutions in greater detail. Based on the result derived in Theorem \ref{Thm_Blow_up}, we have obtained the following upper bound estimate for the lifespan:
$$
T_{\varepsilon}\leq T_{\varepsilon,w}\lesssim \varepsilon^{-\frac{2(p-1)}{2+\gamma-(p-1)(n-q)}}.
$$
Let $T_{\varepsilon,m}$ denote the lifespan of the corresponding mild solution $u=u(t,x)$ to the problem \eqref{Semilinear_Damped_Waves}. Then, we have the next result.
\begin{proposition}[\textbf{Lower bound of lifespan}]\label{Thm_Lower_Bound}
Let $(u_0,u_1)\in \mathcal{A}:=(H^1\cap \mathcal{Y}^q)\times (L^2\cap \mathcal{Y}^q)$. Assume that $p$ belongs to the range \eqref{G-N_p_crit_Gamma_q} and satisfies \eqref{condition-Blowup}, then there exists a constant $\varepsilon_2$ such that for any $\varepsilon\in(0,\varepsilon_2]$ the lifespan $T_{\varepsilon,m}$ of mild solutions $u=u(t,x)$ to the Cauchy problem \eqref{Semilinear_Damped_Waves} satisfies
$$
T_{\varepsilon, m}\gtrsim \varepsilon^{-\frac{2(p-1)}{2+\gamma-(p-1)(n-q)}}.
$$
\end{proposition}
\begin{proof}[\textbf{Proof of Proposition \ref{Thm_Lower_Bound}.}]
As presented in Remark \ref{Remark_Y^q}, we will use Proposition \ref{Decay_Linear_Positive order}  and the approach and notations in the proof of Theorem \ref{Thm_Global existence} with $\mathcal{R}(u_0, u_1) = -q$.
Specifically, for the subcritical case $1 < p < p_{\rm crit}(n,q,\gamma)$, we have the following estimates:
\begin{equation}\label{Crucial_4}
\|\mathcal{N}[u]\|_{X(T)}\leq C_0\varepsilon \|(u_0, u_1)\|_{\mathcal{A}}+C_0(1+T)^{\alpha(p,n,\gamma)}\|u\|_{X(T)}^p.
\end{equation}
and
\begin{align}\label{Crucial_5}
  \|\mathcal{N}[u]-\mathcal{N}[\bar{u}]\|_{X(T)} \leq C_1 (1+T)^{\alpha(p,n, \gamma)}\|u-\bar{u}\|_{X(T)} \left(\|u\|_{X(T)}^{p-1} + \|\bar{u}\|_{X(T)}^{p-1}\right)  
\end{align}
for all $u, \bar{u} \in X(T)$,
where $$\alpha(p,n,\gamma):=1-\frac{n-q}{2}p+\frac{n-q+\gamma}{2}$$
Next, we note that
$
    M^*:= 2C_0 \|(u_0, u_1)\|_{\mathcal{A}}.
$
Then, if we assume that
\begin{align*}
    \max\{C_0, 2C_1\}(1+T)^{\alpha(p,n,q)} \varepsilon^{p-1}< \frac{M^*}{4}, 
\end{align*}
then we may construct a unique local solution $u = \mathcal{N}[u] \in X(T, M^*\varepsilon)$ thanks to the relations \eqref{Crucial_4}, \eqref{Crucial_5}), and Banach's fixed point theorem. Moreover, it holds
$$
\|u\|_{X(T)} < \frac{3M^*\varepsilon}{4}.
$$
Afterwards, we consider
$$
T^*:=\sup\left\{ T\in[0,T_{\varepsilon,m}) \text{ such that } F(T):=\|u\|_{X(T)}\leq M^*\varepsilon \right\}.
$$
The function $F=F(T)$ is a continuous for any $T\in(0,T_{\varepsilon,m})$ so there exists a time $T_0\in(T^*,T_{\varepsilon,m})$ such that $F(T_0)\leq M^*\varepsilon$, which contradicts the definition of $T^*$. This implies that the following condition must hold:
$$
 \max\{C_0, 2C_1\}(1+T^*)^{\alpha(p,n,q)} \varepsilon^{p-1} \geq \frac{M^*}{4}, 
$$
that is,
$$
T_{\varepsilon,m} \geq T^* \gtrsim \varepsilon^{-\frac{2(p-1)}{2+\gamma-(p-1)(n-q)}}.
$$
This completes the proof of Proposition \ref{Thm_Lower_Bound}.
\end{proof}
\begin{remark}
\rm{
Since $u(t, x)$ constructed in Theorem \ref{Thm_Global existence} is a mild solution to \eqref{Semilinear_Damped_Waves}, a standard density argument (see, for instance, Proposition 3.1 in \cite{Ikeda2013}) ensures that this mild solution also qualifies as a weak solution to \eqref{Semilinear_Damped_Waves}. In particular, its maximal existence time satisfies $T_{\varepsilon,m}\leq T_{\varepsilon}$. By combining the upper bound established in Proposition \ref{upperbound} with the lower bound derived in Proposition \ref{Thm_Lower_Bound}, we obtain a new sharp estimate for $T_{\varepsilon}$ in the subcritical regime $1<p<p_{\mathrm{crit}}(n, q, \gamma)$, that is,
\begin{equation}\label{Sharp_lifespan_estimate}
T_{\varepsilon}\sim \varepsilon^{-\frac{2(p-1)}{2+\gamma-(p-1)(n-q)}} \quad \text{ for $0\leq \gamma<q<\frac{n}{2}$}.
\end{equation}
}
\end{remark}
\begin{remark} \label{remark_sharp}
{\rm
Estimate \eqref{Sharp_lifespan_estimate} is genuinely sharp. To the best of the authors' knowledge, no sharp lifespan estimate has been obtained under the assumption that the initial data belong to $L^m$ or $\dot{H}^{-\sigma}$ (see \cite{Ikeda2019,DuongDao2026}, where only non-sharp lifespan estimates are established).
}
\end{remark}

\section*{Acknowledgments}
Trung Loc Tang was funded by the PhD Scholarship Programme of Vingroup Innovation Foundation (VINIF), code VINIF.2025.TS01. Dinh Van Duong was supported by Vietnam Ministry of Education and Training and Vietnam Institute for Advanced Study in Mathematics under grant number B2026-CTT-04. Duc An Phan sincerely acknowledges the financial support provided by the Banking Academy of Vietnam.

\appendix

\section{Some tools from Harmonic Analysis}

\begin{proposition}[Fractional Gagliardo-Nirenberg inequality, \cite{Hajaiej2011}]\label{fractionalgagliardonirenbergineq} 
Let $1<p,\,p_0,\,p_1<\infty$, $a >0$ and $\theta\in [0,a)$. Then, it holds
$$ \|u\|_{\dot{H}^{\theta}_p} \lesssim \|u\|_{L^{p_0}}^{1-\omega(\theta,a)}\, \|u\|_{\dot{H}^{a}_{p_1}}^{\omega(\theta,a)}, $$
where $\omega(\theta,a) =\displaystyle\frac{\frac{1}{p_0}-\frac{1}{p}+\frac{\theta}{n}}{\frac{1}{p_0}-\frac{1}{p_1}+\frac{a}{n}}$ and $\displaystyle\frac{\theta}{a}\leq \omega(\theta,a) \leq 1$.
\end{proposition}
\begin{proposition}[Hardy-Littlewood-Sobolev inequality, \cite{Lieb1983}]\label{Hardy-Littlewood-Sobolev}
Let $0<\gamma<n$ and $1<\eta_2<\eta_1<\infty$ such that $1/\eta_1=1/\eta_2-\gamma/n$. Then, there exists a constant $C$ depending only on $\eta_2$ such that
$$
\|\mathcal{I}_\gamma(f)\|_{L^{\eta_1}}\leq C\|f\|_{L^{\eta_2}}.
$$
\end{proposition}

\section{Auxiliary lemmas}
\begin{lemma} \label{non-empty}
The set of initial data in Theorem \ref{Thm_Blow_up}, i.e. $(u_0,u_1)\in\mathcal{Y}^q\times \mathcal{Y}^q$ with the condition \eqref{assumption_initial_data}, is non-empty.
\end{lemma}
\begin{proof}[\textbf{Proof of Lemma \ref{non-empty}.}]
To prove this lemma, we choose the initial data as follows:
$$
u_0(x)=u_1(x)=C\langle x\rangle^{-n+q},
$$
where $C > 0$.
It is clear that $u_0,u_1\in L^{\frac{n}{n-q}+\epsilon}(\mathbb{R}^n)$, where $\epsilon>0$ is chosen such that $1<n/(n-q)+\epsilon<2$. Next, we recall the formula for the modified Bessel function as follows:
$$
\mathfrak{F}(f)(\xi)=C\int_0^{\infty} g(r) r^{n-1} \tilde{J}_{\frac{n}{2}-1}(r|\xi|) d r \quad  \text{ with } g(|x|):=f(x) \in L^m, m \in [1, 2].
$$
Applying the above formula, we obtain
$$
\widehat{u_1}(\xi)=C \int_0^{\infty} \langle r\rangle^{-n+q}r^{n-1} \tilde{J}_{\frac{n}{2}-1}(r|\xi|) d r.
$$
Therefore, we observe that
$$
\|u_1\|_{\mathcal{Y}^q}=\sup_{\xi \in \mathbb{R}^n}\left\{|\xi|^q |\widehat{u_1}(\xi)|\right\} \sim \sup_{\xi \in \mathbb{R}^n}\left\{|\xi|^q \left|\int_0^{\infty} \langle r\rangle^{-n+q}r^{n-1} \tilde{J}_{\frac{n}{2}-1}(r|\xi|) d r\right|\right\}.
$$
It suffices to prove that $u_1 \in\mathcal{Y}^q$, the case of $u_0$ can be treated in exactly the same way as $u_1$. We also divide our interest two cases as follows:
\begin{itemize}[leftmargin=*]
        
\item If $r|\xi| \leq 1$, we see that
$$
|\xi|^q \left|\int_0^{\frac{1}{|\xi|}} \langle r\rangle^{-n+q}r^{n-1} \tilde{J}_{\frac{n}{2}-1}(r|\xi|) d r\right| \lesssim |\xi|^q \int_0^{\frac{1}{|\xi|}} \langle r\rangle^{-n+q}r^{n-1} d r.
$$
For $|\xi| \in (0, 1]$, we obtain
$$
|\xi|^q \int_0^{\frac{1}{|\xi|}} \langle r\rangle^{-n+q}r^{n-1} d r\lesssim |\xi|^q \int_0^1 r^{n-1} d r+|\xi|^q \int_1^{\frac{1}{|\xi|}} r^{q-1} d r \lesssim 1.
$$
For $|\xi| \in (1, \infty)$, one has
$$
|\xi|^q \int_0^\frac{1}{|\xi|} \langle r\rangle^{-n+q}r^{n-1} d r\lesssim \int_0^\frac{1}{|\xi|}r^{n-q-1} d r
\lesssim \int_0^1r^{n-q-1} d r \lesssim 1.
$$
\item If $r |\xi| >1$, using the first-order asymptotic expansion of the Bessel function $J_\nu(r|\xi|)$ with $\nu=n/2-1$ in the region $r|\xi|>1$, it yields
\begin{align*}
&|\xi|^q\int_\frac{1}{|\xi|}^\infty \langle r\rangle^{-n+q} r^{n-1}\tilde{J}_{\frac{n}{2}-1}(r|\xi|) d r=|\xi|^{q+1-\frac{n}{2}} \int_\frac{1}{|\xi|}^\infty \langle r\rangle^{-n+q} r^\frac{n}{2} J_{\frac{n}{2}-1}(r|\xi|) d r \\
&\quad=|\xi|^{q+1-\frac{n}{2}} \int_\frac{1}{|\xi|}^\infty \langle r\rangle^{-n+q} r^\frac{n}{2} o\left((r|\xi|)^{-\frac{3}{2}}\right) d r \\
&\qquad+C|\xi|^{q+1-\frac{n}{2}} \int_\frac{1}{|\xi|}^\infty \langle r\rangle^{-n+q} r^\frac{n}{2} (r|\xi|)^{-\frac{1}{2}} \cos\left(r|\xi|-\left(\frac{n}{4}-\frac{1}{2}\right)\pi-\frac{\pi}{4}\right) d r:= I_1(\xi)+ I_2(\xi).
\end{align*}
Now we focus on the first integral as follows.

\noindent For $|\xi| \in (0, 1]$, we get
\begin{align*}
|I_1(\xi)|=&|\xi|^{q+1-\frac{n}{2}} \int_\frac{1}{|\xi|}^\infty \langle r\rangle^{-n+q}r^\frac{n}{2} o\left((r|\xi|)^{-\frac{3}{2}}\right) d r\lesssim |\xi|^{q+1-\frac{n}{2}} \int_\frac{1}{|\xi|}^\infty r^{q-\frac{n}{2}} o\left((r|\xi|)^{-\frac{3}{2}}\right) d r \\
\lesssim & |\xi|^{q-\frac{n}{2}-\frac{1}{2}} \int_\frac{1}{|\xi|}^\infty r^{q-\frac{n}{2}-\frac{3}{2}} d r \lesssim 1.
\end{align*}
For $|\xi| \in (1, \infty)$, we arrive at
\begin{align*}
|I_1(\xi)|=&|\xi|^{q+1-\frac{n}{2}} \int_\frac{1}{|\xi|}^\infty \langle r\rangle^{-n+q}r^\frac{n}{2} o\left((r|\xi|)^{-\frac{3}{2}}\right) d r \\
=&|\xi|^{q+1-\frac{n}{2}} \left(\int_\frac{1}{|\xi|}^1 \langle r\rangle^{-n+q}r^\frac{n}{2} o\left((r|\xi|)^{-\frac{3}{2}}\right) d r+\int_1^\infty \langle r\rangle^{-n+q} r^\frac{n}{2} o\left((r|\xi|)^{-\frac{3}{2}}\right) d r\right) \\
\lesssim & |\xi|^{q+1-\frac{n}{2}} \int_\frac{1}{|\xi|}^1 r^\frac{n}{2} o\left((r|\xi|)^{-\frac{3}{2}}\right) d r +|\xi|^{q+1-\frac{n}{2}} \int_1^\infty r^{q-\frac{n}{2}}  o\left((r|\xi|)^{-\frac{3}{2}}\right) d r \\
\lesssim & |\xi| \int_\frac{1}{|\xi|}^1 r^\frac{n}{2} o\left((r|\xi|)^{-\frac{3}{2}}\right) d r +|\xi| \int_1^\infty  o\left((r|\xi|)^{-\frac{3}{2}}\right) d r \\
\lesssim & |\xi|^{-\frac{1}{2}} \int_\frac{1}{|\xi|}^1  r^\frac{n-3}{2} d r +|\xi|^{-\frac{1}{2}} \int_1^\infty r^{-\frac{3}{2}} d r\lesssim 1.
\end{align*}
Next, for the integral $I_2(\xi)$, we see that
$$
I_2(\xi)=|\xi|^{q+\frac{1}{2}-\frac{n}{2}} \int_{\frac{1}{|\xi|}}^{\infty}\langle r\rangle^{-n+q}r^{\frac{n-1}{2}} \cos \left(r|\xi|-\left(\frac{n}{4}-\frac{1}{2}\right) \pi-\frac{\pi}{4}\right) d r.
$$
Applying the change of variables $\omega :=r|\xi|$, we have
\begin{align*}
|I_2(\xi)|=&|\xi|^{q+\frac{1}{2}-\frac{n}{2}} \left|\int_1^{\infty}\left\langle\frac{\omega}{|\xi|}\right\rangle^{-n+q}\left(\frac{\omega}{|\xi|}\right)^{\frac{n-1}{2}} \cos (\omega-\theta) \frac{d \omega}{|\xi|}\right| \\
= & |\xi|^{-n+q} \left|\int_1^{\infty}\omega^{q-\frac{n+1}{2}}|\xi|^{n-q}\cos (\omega-\theta) d \omega\right|=\left|\int_1^{\infty}\omega^{\alpha}\cos (\omega-\theta) d \omega\right| \lesssim 1, 
\end{align*}
where $\theta:=\left(\dfrac{n}{4}-\dfrac{1}{2}\right) \pi+\dfrac{\pi}{4}$ and $\alpha:=q-\dfrac{n+1}{2}<-\dfrac{1}{2}$.
\end{itemize}
Therefore, our proof is finished.
\end{proof}

\begin{lemma}[Lemma A.1, \cite{Dao2019}] \label{prop_Integral inequality}
Let $\alpha, \beta \in \mathbb{R}$. Then, the following inequality holds:
$$
\int_{0}^{t}(1+t-\tau)^{-\alpha}(1+\tau)^{-\beta} \; d \tau \lesssim \begin{cases}(1+t)^{-\min \{\alpha, \beta\}} & \text { if } \max \{\alpha, \beta\}>1, \\ (1+t)^{-\min \{\alpha, \beta\}} \log (e+t) & \text { if } \max \{\alpha, \beta\}=1, \\ (1+t)^{1-\alpha-\beta} & \text { if } \max \{\alpha, \beta\}<1.
\end{cases}
$$
\end{lemma}
\begin{lemma}[Propositions 4.1 and 4.2, \cite{DAbbicco2017}] \label{lemma:1.3}
Let $n \geq 1, 1 \leq m \leq 2$ and $s \geq 0$. For $j=0,1$, assume that $v_0=0$ and $v_1\in H^{[s+j-1]^+} \cap L^m$ in \eqref{Linear_Damped_Waves}, then the solution to \eqref{Linear_Damped_Waves} satisfies the following estimate:
$$
\big\|\partial_t^j v(t, \cdot)\big\|_{\dot{H}^s} \lesssim (1+t-\tau)^{-\frac{n}{2}\left(\frac{1}{m}-\frac{1}{2}\right)-\frac{s}{2}-j} \|v_1\|_{H^{[s+j-1]^+} \cap L^m}. $$
\end{lemma}


\end{document}